\documentclass[reqno]{amsart}

\usepackage{amsmath}
\usepackage{mathtools}
\usepackage{amsthm,amssymb,color,comment}
\usepackage{bbm}
\usepackage{hyperref}
\usepackage[TS1,T1]{fontenc}
\usepackage{inputenc}
\usepackage{dsfont}
\usepackage{tikz}
\usepackage{enumitem}
\usepackage{cite}

\numberwithin{equation}{section}
\numberwithin{subsection}{section}

\theoremstyle{plain}
\newtheorem{thm}{Theorem}[section]

\newtheorem{lem}[thm]{Lemma}
\newtheorem{prop}[thm]{Proposition}

\theoremstyle{definition}
\newtheorem{defn}[thm]{Definition}
\theoremstyle{remark}

\newcommand{\norm}[1]{\left\Vert#1\right\Vert}
\newcommand{\abs}[1]{\left\vert#1\right\vert}
\newcommand{\set}[1]{\left\{#1\right\}}
\newcommand{\R}{\mathbb R}

\newcommand{\eps}{\varepsilon}

\newcommand{\diverg}{\operatorname{div}}

\title[Energy rigidity and weak-strong uniqueness]{Energy rigidity and weak-strong uniqueness for the 2D anisotropic Navier--Stokes equations}

\author{Josef Demmel
}
\address{Department of Mathematics, Friedrich-Alexander-Universit\"at Erlangen-N\"urnberg, Cauerstr.~11, 91058 Erlangen, Germany
}
\email{josef.demmel@fau.de
}
\author{Emil Wiedemann
}
\address{Department of Mathematics, Friedrich-Alexander-Universit\"at Erlangen-N\"urnberg, Cauerstr.~11, 91058 Erlangen, Germany
}
\email{emil.wiedemann@fau.de
}
\date{\today}


\begin{document}

\begin{abstract}
    In two dimensions, we show that dissipation in one spatial direction is sufficient to enforce the energy equality for every weak solution at the natural energy level. In particular, neither anomalous energy loss nor creation can occur. The main difficulty is that the missing directional regularity prevents the usual self-testing argument. We overcome this obstruction through two observations: The pressure is square-integrable by a directional Riesz-transform estimate, and the less regular component is still a renormalized solution. As an application of energy rigidity, we derive a weak-strong uniqueness principle.
\end{abstract}

\maketitle

\section{Introduction}
We consider the anisotropic, incompressible Navier--Stokes equations on $(0,T)\times\R^2$, which take the form
\begin{equation}\label{ans}
    \left\{
	\begin{aligned}
		&\partial_{t}u + \diverg(u\otimes u) + \nabla p = \partial_{x_1}^2 u, \\
		&\diverg u = 0,
	\end{aligned}\right.
\end{equation}
where $u:\R^2\times(0,T)\to\R^2$ is the velocity and $p:\R^2\times(0,T)\to\R$ the pressure. Since diffusion acts in only one direction, it can be seen as an intermediate model between the Euler equations, which have no dissipation ($=0$), and the (classical isotropic) Navier--Stokes equations, which have full dissipation ($=\Delta u$). From a purely mathematical viewpoint, this makes the analysis interesting, as one may ask whether certain properties of the Navier--Stokes equations are still true for the anisotropic system while they fail for the Euler equations. A foundational mathematical treatment of fluids with anisotropic dissipation was given by Chemin et al. \cite{chemin}. From a physical viewpoint, for instance, the anisotropic model has applications as an idealized prototype for oceanography, where additional terms are usually present that model effects like the rotation of the earth or wind forcing \cite{desjardins, grenier}. Moreover, \eqref{ans} is a special case of certain anisotropic Boussinesq systems (for temperature $\theta\equiv0$), which have received considerable attention in the past (see, for example, \cite{cao, danchin, larios, li, paicu}).\\

In analogy to the isotropic Navier--Stokes equations, testing formally with $u$ yields a natural anisotropic energy equality
\begin{equation}\label{energy-equality}
    \norm{u(t)}_{L^2(\R^2)}^2+2\int_s^t\norm{\partial_{x_1}u(\tau)}_{L^2(\R^2)}^2\;d\tau=\norm{u(s)}_{L^2(\R^2)}^2\text{ for every }0\leq s< t\leq T.
\end{equation}
Therefore, this motivates the study of weak solutions in the corresponding energy class. 
\begin{defn}\label{def-weak-sol}
    We call a vector field $u$ a weak solution to \eqref{ans} if 
    \begin{equation}\label{leray-hopf-regularity}
        u\in L^\infty(0,T;L^2(\R^2)),\quad\quad\partial_{x_1}u\in L^2(0,T;L^2(\R^2)),
    \end{equation}
    it is weakly divergence-free, i.e.,
    \begin{equation*}
        \int_{\R^2}u\cdot\nabla\psi\;dx=0
    \end{equation*}
    holds for all $\psi\in C^\infty_c(\R^2)$ and almost every $t$, and
    \begin{equation}\label{distributional-solution}
        \int_0^T \int_{\R^2}\partial_t\phi\cdot u + \nabla\phi:(u\otimes u)-\partial_{x_1}\phi\cdot \partial_{x_1}u \;dx\;dt=0    \end{equation}
    holds for every divergence-free $\phi\in C^\infty_c((0,T)\times\R^2)$.
\end{defn}
At the natural energy level, the regularity for the two-dimensional isotropic Navier--Stokes equations enforces the energy equality. In contrast, weak solutions to the Euler equations do not necessarily satisfy any energy law~\cite{delellis, scheffer}. Our main result now shows that the anisotropic system retains this energy rigidity: every weak solution in the sense of Definition \ref{def-weak-sol} satisfies the anisotropic energy equality.
\begin{thm}[Energy rigidity]\label{thm-energy-rigidity}
    Let $u$ be a weak solution to \eqref{ans}. Then $u$ has a unique representative in $C([0,T];L^2(\R^2))$, which satisfies the anisotropic energy equality \eqref{energy-equality}.
\end{thm}
At first glance, this might seem as an obvious implication, but there is a catch. In contrast to the isotropic system, we have insufficient regularity in our setting to test the system with itself. Although we will show that \eqref{ans} can be interpreted with values in $H^{-1}(\R^2))$, the usual duality argument cannot be applied. Indeed, from the divergence constraint we know $\partial_{x_2}u_2=-\partial_{x_1}u_1\in L^2(\R^2)$, but we lack information about $\partial_{x_2}u_1$ and therefore cannot guarantee that $u$ belongs to $H^1(\R^2)$. Standard approximation schemes, like Galerkin, do not bypass this difficulty as they lack compactness from the missing control of $\partial_{x_2}u_1$ (see also the discussion in \cite{danchin}). \\

So far, we have not specified initial data. The natural choice is solenoidal $u_0\in L^2(\R^2)$. But in this case neither existence nor uniqueness, for arbitrary datum, is currently known due to the aforementioned difficulties. One way to overcome this obstacle is to consider more regular initial data. Existence and uniqueness have been shown in \cite{liang, zhou}, among other things, if one assumes the lacking regularity in the initial data, i.e., $\partial_{x_2}u_{0}\in L^2(\R^2)$ (Zhou and Wu work with the stronger assumption $u_{0}\in H^1(\R^2)$). Related well-posedness results, for essentially the same initial condition, have recently been shown in \cite{cao-guo} on channel domains with hard-wall boundary conditions. Within the isotropic scaling, Zhou and Wu \cite{zhou} even argue that $H^1$-regularity is heuristically sharp for uniqueness. Our second result, which follows from the anisotropic energy equality \eqref{energy-equality}, shows that it suffices to have the additional regularity, $\partial_{x_2}u_1\in L^2(0,T;L^2(\R^2))$, in one of the two competing solutions. More precisely, we show the following weak-strong uniqueness principle.
\begin{thm}[Weak-strong uniqueness]\label{weak-strong-uniqueness}
    Let $u,U$ be weak solutions to \eqref{ans} with \\$\partial_{x_2}U_1\in L^2(0,T;L^2(\R^2))$. If $u(0)=U(0) \text{ in }L^2(\R^2)$, then
    \begin{equation*}
        u(t)=U(t) \text{ in }L^2(\R^2) \text{ for every } t\in(0,T].
    \end{equation*}
\end{thm}
This means that the $H^1$-regularity solution is unique among the potentially much larger class of weak solutions. For an overview of weak-strong uniqueness results in fluid dynamics, we refer to \cite{wiedemann}.
\subsection{Proof strategy}
We start with Theorem \ref{thm-energy-rigidity}. The general idea is not to show the missing regularity $\partial_{x_2}u_1\in L^2(0,T;L^2(\R^2))$, it likely does not even hold in general, instead we work out a way such that it is not required. The first milestone is that we will show there exists a unique pressure $p\in L^2(0,T;L^2(\R^2))$ for every weak solution $u$ of \eqref{ans}. The idea is to consider the Poisson equation for the pressure, $p=R_1^2(u_1^2)+2R_1R_2(u_1u_2)+R_2^2(u_2^2)$, where $R_i$ is the Riesz transform in the $i$-th variable, and prove that the right hand-side is in $L^2(0,T;L^2(\R^2))$. The key tool here is a directional Riesz estimate for $u_1$,
\begin{equation*}
    \norm{R_1(u_1^2)}_{L^2(\R^2)}\leq C\norm{u_1}_{L^2(\R^2)}\norm{\partial_{x_1}u_1}_{L^2(\R^2)}.
\end{equation*}
With square-integrable pressure, we will show the energy equality \eqref{energy-equality} component by component:
\begin{equation*}
    \norm{u_2(t)}_{L^2(\R^2)}^2+2\int_s^t\norm{\partial_{x_1}u_2(\tau)}_{L^2(\R^2)}^2\;d\tau=\norm{u_2(s)}_{L^2(\R^2)}^2-2\int_s^t\int_{\R^2}p\partial_{x_1}u_1\;dx\;d\tau,
\end{equation*}
and
\begin{equation*}
    \norm{u_1(t)}_{L^2(\R^2)}^2+2\int_s^t\norm{\partial_{x_1}u_1(\tau)}_{L^2(\R^2)}^2\;d\tau=\norm{u_1(s)}_{L^2(\R^2)}^2+2\int_s^t\int_{\R^2}p\partial_{x_1}u_1\;dx\;d\tau.
\end{equation*}
For $u_2$, we have control over the whole gradient. Therefore, we can use the formal argument by testing the equation with itself. For $u_1$, the central new observation is that it is still a renormalized solution; see Proposition \ref{renormalizatin-u-1} for the precise definition. To prove this, we establish a commutator estimate and exploit the square-integrability of the pressure. With the renormalization property, the energy balance for $u_1$ results from a dominated convergence argument. Finally, the unique representative $u\in C([0,T];L^2(\R^2))$ is a consequence of the energy equality.\\

The proof of Theorem \ref{weak-strong-uniqueness} is a classical comparison argument for $w:=u-U$. The crucial tool here is the just proven energy equality \eqref{energy-equality} for weak solutions. The additional regularity of $U$ is sufficient to control the nonlinear terms. The resulting relative energy inequality closes by Gr\"onwall's lemma.
\subsection*{Declaration of AI Use} During the research, the authors used OpenAI’s GPT-5.5 Pro and GPT-5.6 Sol to generate candidate proof strategies and intermediate arguments for Theorem \ref{thm-energy-rigidity}. The authors independently reconstructed and verified every step, checked all references against the original sources, wrote the manuscript themselves, and accept full responsibility for its mathematical content.
\section{Preliminaries}
For $p,q\in[1,\infty]$ and $s\in\R$, we will write $L_t^p=L^p(0,T)$, $L_x^q=L^q(\R^2)$, $H^s_x=H^s(\R^2)$ and $L^1_{loc}=L^1_{loc}((0,T)\times\R^2)$. To begin, we gather some implications solely based on the regularity class.
\begin{lem}\label{prelim-regularity}
    Let a vector field $v$ satisfy \eqref{leray-hopf-regularity} and be divergence-free. Then
    \begin{align*}
        &\quad v_2\in L^2_t H^1_x,\quad v_2\in L^4_t L^4_x\quad\text{and}\quad v_1v_2\in L^2_t L^2_x.
        \end{align*}
\end{lem}
\begin{proof}
    From the divergence constraint, we have
    \begin{equation*}
        \partial_{x_2}v_2=-\partial_{x_1}v_1 \in L^2_t L^2_x,
    \end{equation*}
    and thus $v_2\in L^2_tH^1_x$. With Ladyzhenskaya's inequality, we immediately get $v_2\in L^4_tL^4_x$. Regarding the mixed integrability, we use the one-dimensional Gagliardo--Nirenberg inequality \ref{gagliardo} to obtain
    \begin{align*}
    \norm{v_1v_2}_{L^2_tL^2_x}\leq\norm{\norm{v_1}_{L^\infty_{x_1}L^2_{x_2}}\norm{v_2}_{L^2_{x_1}L^\infty_{x_2}}}_{L^2_t}&\leq 2\norm{\norm{v_1}^{1/2}_{L^2_x}\norm{\partial_{x_1}v_1}^{1/2}_{L^2_x}\norm{v_2}^{1/2}_{L^2_x}\norm{\partial_{x_2}v_2}^{1/2}_{L^2_x}}_{L^2_t}\\
    &\leq 2\norm{v_1}^{1/2}_{L^\infty_tL^2_x}\norm{v_2}^{1/2}_{L^\infty_tL^2_x}\norm{\partial_{x_1}v_1}^{1/2}_{L^2_tL^2_x}\norm{\partial_{x_2}v_2}^{1/2}_{L^2_tL^2_x}.
    \end{align*}
\end{proof}
\section{Square-integrable pressure}
The purpose of this section is to show that the pressure $p$ associated to a weak solution $u$ of \eqref{ans} is in fact in $L^2_tL^2_x$. As in the isotropic case, we can recover the pressure $p$ from the velocity $u$ by the Poisson equation
\begin{equation}\label{poisson}
    -\Delta p=\sum_{i,j=1}^2\partial_{x_i}\partial_{x_j}(u_i u_j),
\end{equation}
which arises by taking the divergence of the momentum equation in \eqref{ans}. Using the Riesz transform (see Lemma \ref{riesz}), we can (formally) give a solution $p$ to \eqref{poisson} by 
\begin{equation}\label{riesz-equation}
    p = \sum_{i,j=1}^2R_iR_j(u_iu_j)= R_1^2(u_1^2)+2R_1R_2(u_1u_2)+R_2^2(u_2^2).
\end{equation}
If the right hand-side of \eqref{riesz-equation} is in $L^2_t L^2_x$, this is the unique solution to \eqref{poisson} with $p\in L^2_tL^2_x$ (see, for more details, \cite[Section 5]{robinson}). With Lemma \ref{prelim-regularity} and the boundedness of the Riesz transform, we can control the last two terms. However, it is not clear whether, in general, $u_1^2$ is in $L^2_{x}$ or not. Thus, we need a refined directional Riesz estimate.
\begin{lem}\label{directional-riesz}
    Let $f,\partial_{x_1}f\in L^2_x$. Then there exists a constant $C>0$, such that 
    \begin{equation*}
        \norm{R_1(f^2)}_{L^2_x}\leq C\norm{f}_{L^2_x}\norm{\partial_{x_1}f}_{L^2_x}.
    \end{equation*}
\end{lem}
\begin{proof}
    For simplicity, we first assume $f\in\mathcal{S}(\R^2)$. To begin, we want to consider the Fourier transform of $f^2$, which we will denote by $\widehat{f^2}$ (see \eqref{fourier}). With the Fourier product-convolution formula and Cauchy--Schwarz, a first key observation is that it has an upper bound independent of the vertical frequency 
    \begin{align*}
        \abs{\widehat{f^2}(\xi_1,\xi_2)}&=\frac{1}{2\pi}\abs{\int_\R\int_\R\widehat{f}(\eta_1,\eta_2)\widehat{f}(\xi_1-\eta_1,\xi_2-\eta_2)\;d\eta_2\;d\eta_1}\\
        &\leq\frac{1}{2\pi} \int_\R\left(\int_\R\abs{\widehat{f}(\eta_1,\eta_2)}^2\;d\eta_2\right)^{1/2}\left(\int_\R\abs{\widehat{f}(\xi_1-\eta_1,\xi_2-\eta_2)}^2\;d\eta_2\right)^{1/2} \;d\eta_1\\
        &=\frac{1}{2\pi}\int_\R\left(\int_\R\abs{\widehat{f}(\eta_1,\eta_2)}^2\;d\eta_2\right)^{1/2}\left(\int_\R\abs{\widehat{f}(\xi_1-\eta_1,\zeta_2)}^2\;d\zeta_2\right)^{1/2} \;d\eta_1
        \leq \frac{1}{2\pi}(a*a)(\xi_1), 
    \end{align*}
    with $a(\xi_1):=\left(\int_\R\abs{\widehat{f}(\xi_1,\xi_2)^2}\;d\xi_2\right)^{1/2}$. By Plancherel, we have the sizes
    \begin{equation*}
        \norm{a}^2_{L^2_{\xi_1}}=\norm{\widehat{f}}^2_{L^2_{\xi}}=\norm{f}_{L^2_x}^2,\quad\norm{\xi_1a}^2_{L^2_{\xi_1}}=\int_\R\xi_1^2\int_\R\abs{\widehat{f}(\xi_1,\xi_2)}^2\;d\xi_2\;d\xi_1=\norm{\xi_1\widehat{f}}^2_{L^2_{\xi}}=\norm{\partial_{x_1}f}_{L^2_x}^2.
    \end{equation*}
    Using the inequality above, we get for the Riesz transform
    \begin{align*}
        \norm{R_1(f^2)}_{L^2_x}^2&=\int_\R\int_\R\frac{\xi_1^2}{\xi_1^2+\xi_2^2}\abs{\widehat{f^2}(\xi_1,\xi_2)}^2\;d\xi_1\;d\xi_2 \leq \frac{1}{4\pi^2}\int_\R\int_\R\frac{\xi_1^2}{\xi_1^2+\xi_2^2}\abs{(a*a)(\xi_1)}^2\;d\xi_1\;d\xi_2\\
        &=\frac{1}{4\pi^2}\int_\R\abs{(a*a)(\xi_1)}^2\int_\R\frac{\xi_1^2}{\xi_1^2+\xi_2^2}\;d\xi_2\;d\xi_1=\frac{1}{4\pi^2}\int_\R\abs{(a*a)(\xi_1)}^2\abs{\xi_1}\int_\R\frac{1}{1+s^2}\;ds\;d\xi_1\\
        &=\frac{1}{4\pi}\int_\R\abs{(a*a)(\xi_1)}^2\abs{\xi_1}\;d\xi_1.
    \end{align*}
    Next, we want to go back to the physical variables. Therefore, we introduce
    \begin{equation*}
        A := \mathcal{F}^{-1}_{1}a,
    \end{equation*}
    where $\mathcal{F}^{-1}_{1}$ is the unitary one-dimensional inverse Fourier transform with respect to the first variable. Using a similar argument as before and the Gagliardo--Nirenberg inequality \ref{gagliardo}, we conclude
    \begin{align*}
        \int_\R\abs{(a*a)(\xi_1)}^2\abs{\xi_1}\;d\xi_1&=2\pi\int_\R\abs{\widehat{A^2}(\xi_1)}^2\abs{\xi_1}\;d\xi_1= 2\pi\int_\R\abs{\widehat{A^2}(\xi_1)}\left(\abs{\widehat{A^2}(\xi_1)}\abs{\xi_1}\right)\;d\xi_1\\
        &\leq 2\pi\left(\int_\R\abs{\widehat{A^2}(\xi_1)}^2\;d\xi_1\right)^{1/2}\left(\int_\R\abs{\widehat{A^2}(\xi_1)}^2\xi_1^2\;d\xi_1\right)^{1/2}\\
        &=2\pi\norm{\widehat{A^2}}_{L^2_{\xi_1}}\norm{\xi_1\widehat{A^2}}_{L^2_{\xi_1}}=2\pi\norm{A^2}_{L^2_{x_1}}\norm{(A^2)'}_{L^2_{x_1}}=4\pi\norm{A}^2_{L^4_{x_1}}\norm{A'A}_{L^2_{x_1}}\\
        &\leq 4\pi\norm{A}^2_{L^4_{x_1}}\norm{A}_{L^\infty_{x_1}}\norm{A'}_{L^2_{x_1}}\\&\leq C\left(\norm{A}^{3/2}_{L^2_{x_1}}\norm{A'}^{1/2}_{L^2_{x_1}}\right)\left(\norm{A}^{1/2}_{L^2_{x_1}}\norm{A'}^{1/2}_{L^2_{x_1}}\right)\norm{A'}_{L^2_{x_1}}\\
        &=C\norm{A}^{2}_{L^2_{x_1}}\norm{A'}^{2}_{L^2_{x_1}}=C\norm{\widehat{A}}^{2}_{L^2_{\xi_1}}\norm{\xi_1\widehat{A}}^{2}_{L^2_{\xi_1}}=C\norm{a}^{2}_{L^2_{\xi_1}}\norm{\xi_1a}^{2}_{L^2_{x_1}}\\ &=C\norm{f}_{L^2_x}^2\norm{\partial_{x_1}f}_{L^2_x}^2.
    \end{align*}
    To show the general case, remove the Schwartz assumption and choose, by density, a sequence $f_n\in \mathcal{S}(\R^2)$ such that $f_n\to f$ and $\partial_{x_1}f_n\to\partial_{x_1}f$ in $L^2_x$. We get immediately 
    \begin{equation*}
        \norm{f_n^2-f^2}_{L^1_x}\leq\norm{f_n-f}_{L^2_x}\left(\norm{f_n}_{L^2_x}+\norm{f}_{L^2_x}\right)\to 0,
    \end{equation*}
    and therefore $\widehat{f_n^2}\to\widehat{f^2}$ uniformly. With Fatou's lemma, this finishes the general case
    \begin{align*}
        \norm{R_1(f^2)}_{L^2_x}^2&=\int_{\R^2}\abs{i\frac{\xi_1}{\abs{\xi}}\widehat{f^2}(\xi)}^2\;d\xi\leq \liminf_{n\to\infty}\int_{\R^2}\abs{i\frac{\xi_1}{\abs{\xi}}\widehat{f^2_n}(\xi)}^2\;d\xi=\liminf_{n\to\infty}\norm{R_1(f^2_n)}_{L^2_x}^2\\
        &\leq C^2\lim_{n\to\infty}\norm{f_n}_{L^2_x}^2\norm{\partial_{x_1}f_n}_{L^2_x}^2=C^2\norm{f}_{L^2_x}^2\norm{\partial_{x_1}f}_{L^2_x}^2.
    \end{align*}
\end{proof}
Now, we control  every term in \eqref{riesz-equation}, so we state the main result of this section.
\begin{prop}\label{prop-pressure-time-regularity}
    Let $u$ be a weak solution to \eqref{ans}. Then there exists a unique pressure $p\in L^2_tL^2_x$ such that \eqref{ans} holds in the sense of distributions. Moreover, we have the additional regularity \begin{equation}\label{time-regularity}
        \partial_tu\in L^2_tH^{-1}_x
    \end{equation}
\end{prop}
\begin{proof}
    The pressure part follows from \eqref{riesz-equation} together with Lemmas \ref{prelim-regularity} and \ref{directional-riesz}. The time regularity is also a relatively straightforward consequence:
    \begin{align*}
        \norm{\partial_tu}_{L^2_tH^{-1}_x}\leq& \norm{\diverg(u\otimes u)}_{L^2_tH^{-1}_x}+\norm{\nabla p}_{L^2_tH^{-1}_x} + \norm{\partial_{x_1}^2u}_{L^2_tH^{-1}_x}\\
        \leq&\norm{\partial_{x_1}(u_1^2)}_{L^2_tH^{-1}_x}+\norm{\partial_{x_1}(u_1u_2)}_{L^2_tH^{-1}_x}+\norm{\partial_{x_2}(u_1u_2)}_{L^2_tH^{-1}_x}+\norm{\partial_{x_2}(u_2^2)}_{L^2_tH^{-1}_x}\\&+\norm{\nabla p}_{L^2_tH^{-1}_x} + \norm{\partial_{x_1}^2u}_{L^2_tH^{-1}_x}\\
        \leq&\norm{\partial_{x_1}(u_1^2)}_{L^2_tH^{-1}_x}+C\left(\norm{u}_{L^\infty_tL^2_x},\norm{\partial_{x_1}u}_{L^2_tL^2_x}\right),
    \end{align*}
    for the first term, we use the Fourier transform and Lemma \ref{directional-riesz}
    \begin{align*}
        \norm{\partial_{x_1}(u_1^2)}_{H^{-1}_x}^2=\int_{\R^2}\frac{\xi_1^2}{1+\abs{\xi}^2}\abs{\widehat{u_1^2}(\xi)}^2\;d\xi\leq\int_{\R^2}\frac{\xi_1^2}{\abs{\xi}^2}\abs{\widehat{u_1^2}(\xi)}^2\;d\xi=\int_{\R^2}\abs{\widehat{R_1(u_1^2)}(\xi)}^2\;d\xi=\norm{R_1(u_1^2)}_{L^2_x}^2.
    \end{align*}
\end{proof}
For completeness, we recall the following weak continuity property.
\begin{lem}\label{weak-time-continuity}
    A weak solution $u$ of \eqref{ans} has a unique representative in $C_w([0,T];L^2_x)$.
\end{lem}
\begin{proof}
    The existence follows from the standard balance-law argument of \cite[Lemma 8]{delellis1}, which applies verbatim to \eqref{ans}. For uniqueness, assume $\tilde{u}$ is another weakly continuous $L^2_x$-representative. Then, for any $\phi\in L^2_x$, the continuous map
    \begin{equation*}
        t\to \int_{\R^2}(\bar{u}(t)-\tilde{u}(t))\cdot\phi\;dx,
    \end{equation*}
    vanishes almost everywhere, hence everywhere.
\end{proof}
Now, we are in a position to show \eqref{energy-equality} in two steps. First, we show  the energy equality for the vertical component
\begin{equation*}
    \norm{u_2(t)}_{L^2_x}^2+2\int_s^t\norm{\partial_{x_1}u_2(\tau)}_{L^2_x}^2\;d\tau=\norm{u_2(s)}_{L^2_x}^2-2\int_s^t\int_{\R^2}p\partial_{x_1}u_1\;dx\;d\tau.
\end{equation*}
Then, with some additional work, we show the counterpart for the horizontal component
\begin{equation*}
    \norm{u_1(t)}_{L^2_x}^2+2\int_s^t\norm{\partial_{x_1}u_1(\tau)}_{L^2_x}^2\;d\tau=\norm{u_1(s)}_{L^2_x}^2+2\int_s^t\int_{\R^2}p\partial_{x_1}u_1\;dx\;d\tau.
\end{equation*}
Adding the two identities yields \eqref{energy-equality}.
\section{The vertical component}
We have previously seen that the full gradient of $u_2$ is in $L^2_x$, due to the divergence constraint. Essentially, this already implies the energy equality.
\begin{prop}\label{vertical-energy}
    Let $u$ be a weak solution to \eqref{ans}. Then the unique representative of $u_2$ is in $C([0,T];L^2_x)$ and, for every $0\leq s< t\leq T$, satisfies
    \begin{equation*}
    \norm{u_2(t)}_{L^2_x}^2+2\int_s^t\norm{\partial_{x_1}u_2(\tau)}_{L^2_x}^2\;d\tau=\norm{u_2(s)}_{L^2_x}^2-2\int_s^t\int_{\R^2}p\partial_{x_1}u_1\;dx\;d\tau.
\end{equation*}
\end{prop}
\begin{proof}
    For time continuity, we apply \cite[Section 5.9.2]{evans} with $u_2\in L^2_tH^1_x$ and $\partial_tu_2\in L^2_tH^{-1}_x$. The second component of the momentum equation in \eqref{ans} is 
    \begin{equation*}
        \partial_tu_2+\partial_{x_1}(u_1u_2)+\partial_{x_2}(u_2^2)+\partial_{x_2}p=\partial_{x_1}^2u_2.
    \end{equation*}
    From our previous results, we know that every term is in $L^2_tH^{-1}_x$, and thus we can test the equation against $u_2$, integrate it in space, and perform the usual integration by parts to obtain
    \begin{equation*}
        \frac{1}{2}\frac{d}{dt}\norm{u_2(t)}_{L^2_x}^2-\int_{\R^2}u_1u_2\partial_{x_1}u_2\;dx -\int_{\R^2}u_2^2\partial_{x_2}u_2\;dx-\int_{\R^2}p\partial_{x_2}u_2\;dx+\norm{\partial_{x_1}u_2(t)}_{L^2_x}^2=0.
    \end{equation*}
    Next, we want to show that both convective integrals vanish. First, observe that, for almost every time, $u_2^3\in L^1_x$ and $\partial_{x_2}(u_2^3)=3u_2^2\partial_{x_2}u_2\in L^1_x$ and thus
    \begin{equation*}
        \int_{\R^2}u_2^2\partial_{x_2}u_2\;dx=\frac{1}{3}\int_{\R^2}\partial_{x_2}(u_2)^3\;dx=0.
    \end{equation*}
    Similarly,
    \begin{equation*}
        \int_{\R^2}u_1u_2\partial_{x_1}u_2\;dx=-\frac{1}{2}\int_{\R^2}\partial_{x_1}u_1u_2^2\;dx=\frac{1}{2}\int_{\R^2}\partial_{x_2}u_2u_2^2\;dx=0,
    \end{equation*}
    where we also used incompressibility. Again by incompressibility,
    \begin{equation*}
        \int_{\R^2}p\partial_{x_2}u_2\;dx=-\int_{\R^2}p\partial_{x_1}u_1\;dx.
    \end{equation*}
    We conclude by integrating in time.
\end{proof}
\section{The horizontal component}
For $u_1$, we have no information about the vertical derivative, and thus no additional regularity beyond \eqref{leray-hopf-regularity}. Therefore, the previous argument fails, as we cannot test with $u_1$ and also lack integrability. The distinctive and to our knowledge new observation is that $u_1$ is still a renormalized solution, which suffices for the energy equality.
\subsection{The commutator}
Let $\eta\in C^\infty_c(\R^2)$ be a standard even Friedrichs mollifier. We set
\begin{equation*}
    f^{\eps} := \eta_{\eps}*f \text{ with }\eta_\eps(x)=\frac{1}{\eps^2}\eta\left(\frac{x}{\eps}\right) \text{ for }\eps>0.
\end{equation*}
Within our anisotropic setting, we show an analog to the classical commutator estimate of the DiPerna--Lions theory \cite{diperna-lions}. 
\begin{prop}\label{anisotropic-commutator}
    Let a scalar $f\in L^2_x$ and a weakly divergence-free vector field $v\in L^2_x$ satisfy $\partial_{x_1}f, \partial_{x_1}v\in L^2_x$. Then
    \begin{equation*}
        \diverg(f^{\eps}v-(fv)^{\eps})\to 0 \text{ strongly in }L^1_x \text{ for }\eps\to0.
    \end{equation*}
\end{prop}
\begin{proof}
    First, we split the commutator by partial derivatives
    \begin{align*}
        \diverg(f^{\eps}v-(fv)^{\eps})= R_\eps^{(1)}+R_\eps^{(2)} \text{ with }&R_\eps^{(1)}:=v_1\partial_{x_1}f^{\eps}-\partial_{x_1}(v_1f)^{\eps}\text{ and }R_\eps^{(2)}:=v_2\partial_{x_2}f^{\eps}-\partial_{x_2}(v_2f)^{\eps}.
    \end{align*}
    For the first component, commutation of convolutions with derivatives yields
    \begin{align*}
        R_\eps^{(1)}&=v_1\partial_{x_1}f^{\eps}-\partial_{x_1}(v_1f)^{\eps}=v_1(\partial_{x_1}f)^{\eps}-(\partial_{x_1}(v_1f))^{\eps}\\
        &=v_1(\partial_{x_1}f)^{\eps}-(v_1\partial_{x_1}f))^{\eps}-(f\partial_{x_1}v_1))^{\eps}\to-f\partial_{x_1}v_1 \text{ in }L^1_x\text{ for }\eps\to0.
    \end{align*}
    For the second component, recall that we have $v_2\in H^1_x$ by the divergence constraint. This allows us to compensate the mission control of $\partial_{x_2}f$. We use the fundamental theorem of calculus together with substitution to obtain
    \begin{align*}
        R_\eps^{(2)}(x)=&\int_{\R^2}\partial_{y_2}\eta_{\eps}(y)(v_2(x)-v_2(x-y))f(x-y)\;dy\\
        =&\frac{1}{\eps}\int_{\R^2}\partial_{z_2}\eta(z)(v_2(x)-v_2(x-\eps z))f(x-\eps z)\;dz\\
        =&\frac{1}{\eps}\int_{\R^2}\partial_{z_2}\eta(z)\int_0^1 \eps z\cdot\nabla v_2(x-\theta \eps z)\;d\theta f(x-\eps z)\;dz\\
        =&\sum_{i=1}^2\int_{\R^2}\partial_{z_2}\eta(z)z_i\int_0^1\partial_{x_i}v_2(x-\theta\eps z) f(x-\eps z)\;d\theta\;dz.
    \end{align*}
    Now, we want to go to the limit with the dominated convergence theorem. For fixed $z$ and $\theta$, we have, by continuity of translations in $L^2$, that
    \begin{equation*}
        \partial_{x_i}v_2(\cdot-\theta\eps z) f(\cdot-\eps z)\to\partial_{x_i}v_2 f \text{ in }L^1_x\text{ for }\eps\to0.
    \end{equation*}
    Moreover, obviously the product $\partial_{x_i}v_2(\cdot-\theta\eps z) f(\cdot-\eps z)$ is uniformly bounded and $\partial_{z_2}\eta_{\eps}(z)z_i$ is integrable, hence dominated convergence gives
    \begin{equation*}
        R_\eps^{(2)} \to \sum_{i=1}^2\left(\int_{\R^2}\partial_{z_2}\eta(z)z_i\;dz\right)\partial_{x_i}v_2f\text{ in }L^1_x\text{ for }\eps\to0.
    \end{equation*}
    Regarding the kernels, we have with integration by parts
    \begin{equation*}
        \sum_{i=1}^2\left(\int_{\R^2}\partial_{z_2}\eta(z)z_i\;dz\right)=-\sum_{i=1}^2\left(\int_{\R^2}\eta(z)\partial_{z_2}z_i\;dz\right)=-1.
    \end{equation*}
    In summary, we obtain
    \begin{equation*}
        R_\eps^{(1)}+R_\eps^{(2)}\to -(\partial_{x_1}v_1+\partial_{x_2}v_2)f=0\text{ in }L^1_x\text{ for }\eps\to0,
    \end{equation*}
    due to incompressibility.
    
\end{proof}
\subsection{Renormalization}
\begin{prop}\label{renormalizatin-u-1}
    Let $u$ be a weak solution to \eqref{ans}. Then the unique representative of $u_1 \in C_w([0,T];L^2_x)$ is a renormalized solution, i.e., for every $\beta\in C^2(\R)$ with $\beta(0)=\beta'(0)=0$, $\beta',\beta''\in L^\infty(\R)$ and every $0\leq s < t \leq T$, we have
    \begin{equation*}
        \int_{\R^2}\beta(u_1(t))\;dx+\int_s^t\int_{\R^2}\beta''(u_1)(\partial_{x_1}u_1)^2\;dx\;d\tau=\int_{\R^2}\beta(u_1(s))\;dx+\int_s^t\int_{\R^2}p\beta''(u_1)\partial_{x_1}u_1\;dx\;d\tau.
    \end{equation*}
\end{prop}
\begin{proof}
    For $u_1^\eps$, we consider the corresponding momentum equation
    \begin{equation}\label{ans-u-1-eps}
        \partial_tu_1^{\eps}+\diverg(u_1^{\eps}u)+ \partial_{x_1}p_{\eps}-\partial_{x_1}^2u_1^{\eps}= R_\eps,
    \end{equation}
    where $R_\eps:=\diverg(u_1^{\eps}u-(u_{1}u)^{\eps})$. Let us start with some preliminary observations. Recall that $\partial_tu_1\in L^2(0,T;H^{-1}_x)$, and thus we have by standard properties of the convolution that $u_1^{\eps}\in W^{1,2}(0,T;L^2_x)\hookrightarrow C([0,T];L^2_x)$, which is sufficient for the chain rule for Sobolev functions \cite[Theorem 4.4]{evans2}. Before doing so, we need to choose a cutoff $\chi\in C_c^\infty(\R^2)$ with $0\leq\chi\leq1$, $\chi=1$ on $B_1(0)$, and set $\chi_{R}(x)=\chi(x/R)$. Then, we get with the chain rule and the usual integration by parts
    \begin{align*}
        &\int_{\R^2}\chi_R\beta(u_1^{\eps}(t))\;dx -\int_{\R^2}\chi_R\beta(u_1^{\eps}(s))\;dx+ \int_s^t\int_{\R^2}\chi_R\beta''(u_1^{\eps})((\partial_{x_1}u_1^{\eps})^2-p_{\eps}\partial_{x_1}u_1^{\eps})\;dx\;d\tau\\&=\int_s^t\int_{\R^2}\chi_R\beta'(u_1^{\eps})R_{\eps}\;dx\;d\tau+\int_s^t\int_{\R^2}\nabla\chi_R\cdot\left[\beta(u_1^{\eps})u+\beta'(u_1^{\eps})p_\eps e_1-\beta'(u_1^{\eps})\partial_{x_1}u_1^{\eps} e_1\right]\;dx\;d\tau.
    \end{align*}
    By the choice of $\beta$, we have sufficient integrability for the limit $R\to\infty$. Since $\norm{\nabla\chi_R}_{L^\infty_x}\leq \frac{C}{R}$, the second term on the right hand-side vanishes
    \begin{align*}
        \int_{\R^2}\beta(u_1^{\eps}(t))\;dx &-\int_{\R^2}\beta(u_1^{\eps}(s))\;dx+ \int_s^t\int_{\R^2}\beta''(u_1^{\eps})((\partial_{x_1}u_1^{\eps})^2-p_{\eps}\partial_{x_1}u_1^{\eps})\;dx\;d\tau\\&=\int_s^t\int_{\R^2}\beta'(u_1^{\eps})R_{\eps}\;dx\;d\tau.
    \end{align*}
    Finally, we let $\eps\to0$. For the right hand-side, we use Proposition \ref{anisotropic-commutator} and the boundedness of $\beta'$ to deduce
    \begin{equation*}
        \int_s^t\int_{\R^2}\beta'(u_1^{\eps})R_{\eps}\;dx\;d\tau \to 0.
    \end{equation*}
    For the left hand-side, we first consider the endpoint integrals. With the fundamental theorem of calculus, we get, for $r\in\set{s,t}$,
    \begin{align*}
        \norm{\beta(u_1^{\eps}(r))-\beta(u_{1}(r))}_{L^1_x}&= \norm{(u_1^{\eps}(r)-u_{1}(r))\int_0^1\beta'((1-\theta)u_{1}(r)+\theta u_1^{\eps}(r))\;d\theta}_{L^1_x}\\
        &\leq \frac{1}{2}\norm{\beta''}_{L^\infty(\R)}\left(\norm{u_1^{\eps}(r)}_{L^2_x}+\norm{u_{1}(r)}_{L^2_x}\right)\norm{u_1^{\eps}(r)-u_{1}(r)}_{L^2_x}\\
        &\to 0,
    \end{align*}
    as $u_1^{\eps}(r)\to u_{1}(r)$ in $L^2_x$. It remains to pass to the limit in the bulk integral
    \begin{align*}
        &\norm{\beta''(u_1^{\eps})((\partial_{x_1}u_1^{\eps})^2-p_{\eps}\partial_{x_1}u_1^{\eps})-\beta''(u_{1})((\partial_{x_1}u_{1})^2-p\partial_{x_1}u_{1})}_{L^1(s,t;L^1_x)}\\&\leq \norm{[\beta''(u_1^{\eps})-\beta''(u_{1})]((\partial_{x_1}u_{1})^2-p\partial_{x_1}u_{1})}_{L^1(s,t;L^1_x)}\\&+\norm{\beta''}_{L^\infty(\R)}\norm{(\partial_{x_1}u_1^{\eps})^2-(\partial_{x_1}u_{1})^2}_{L^1(s,t;L^1_x)}+\norm{\beta''}_{L^\infty(\R)}\norm{p_{\eps}\partial_{x_1}u_1^{\eps}-p\partial_{x_1}u_{1}}_{L^1(s,t;L^1_x)}\\
        &\leq \norm{[\beta''(u_1^{\eps})-\beta''(u_{1})]((\partial_{x_1}u_{1})^2-p\partial_{x_1}u_{1})}_{L^1(s,t;L^1_x)}\\&+C\norm{\partial_{x_1}u_1^{\eps}+\partial_{x_1}u_{1}}_{L^2(s,t;L^2_x)}\norm{\partial_{x_1}u_1^{\eps}-\partial_{x_1}u_{1}}_{L^2(s,t;L^2_x)}\\&+C\norm{p_{\eps}-p}_{L^2(s,t;L^2_x)}\norm{\partial_{x_1}u_1^{\eps}}_{L^2(s,t;L^2_x)}+C\norm{p}_{L^2(s,t;L^2_x)}\norm{\partial_{x_1}u_1^{\eps}-\partial_{x_1}u_1}_{L^2(s,t;L^2_x)}\\
        &\to 0,
    \end{align*}
    where, for the first term, we used the stability of convergence in measure under continuous compositions.
\end{proof}
\begin{proof}[Proof of Theorem \ref{thm-energy-rigidity}]
    To begin with, we will show the energy equality for the horizontal component. Let $\psi\in C^\infty_c(\R)$ with $0\leq\psi\leq1$, $\psi=1$ on $[-1,1]$, $\psi(z)=\psi(-z)$, and define 
    \begin{equation*}
        \beta''_R(z):=\psi(z/R),\quad \beta'_R(z):=\int_0^z\beta''_R(y)\;dy,\quad\beta_R(z):=\int_0^z\beta'_R(y)\;dy.
    \end{equation*}
    Observe that, $0\leq \beta_R(z)\leq\frac{1}{2}z^2$, $0\leq\beta''_R\leq1$, and, as $R\to\infty$, $\beta_R(z)\to \frac{1}{2}z^2$ and $\beta_R''(z)\to 1$ pointwise. Moreover, $\beta_R$ is a valid renormalization for Proposition \ref{renormalizatin-u-1} and, thus, we get by dominated convergence
    \begin{equation*}
        \frac{1}{2}\norm{u_1(t)}_{L^2_x}^2+\int_s^t\norm{\partial_{x_1}u_1(\tau)}_{L^2_x}^2\;d\tau=\frac{1}{2}\norm{u_1(s)}_{L^2_x}^2+\int_s^tp\int_{\R^2}\partial_{x_1}u_1\;dx\;d\tau,
    \end{equation*}
    which implies the energy equality \eqref{energy-equality} together with Proposition \ref{vertical-energy}. It remains to show the strong time continuity. With the energy equality \eqref{energy-equality}, we know that $t\to\norm{u(t)}_{L^2_x}$ is continuous. Since we also have weak time continuity, shown in Lemma \ref{weak-time-continuity}, this already implies strong time continuity in $L^2_x$. The uniqueness of the representative follows as two continuous functions agreeing almost everywhere agree everywhere.
\end{proof}
\section{Weak-strong uniqueness}
With the energy equality \eqref{energy-equality} at hand, weak-strong uniqueness follows by the classical comparison argument.
\begin{proof}[Proof of Theorem \ref{weak-strong-uniqueness}]
    We want to derive an energy balance for $w:=u-U$, but we need some preliminary work before doing so. To begin, both $u$ and $U$ are weak solutions to \eqref{ans} and thus satisfy the energy equality \eqref{energy-equality} by Theorem \ref{thm-energy-rigidity}. Regarding the cross term, we use again mollification. For $\eps>0$, we consider the space-mollified $u^\eps$ and $U^\eps$. As we have previously seen, we have with Proposition \ref{prop-pressure-time-regularity}
    \begin{equation*}
        u^\eps,U^\eps \in W^{1,2}(0,T;L^2_x).
    \end{equation*}
    By the product rule for Sobolev functions, we get
    \begin{equation}\label{time-product-rule}
        \frac{d}{dt}\int_{\R^2}u^{\eps}\cdot U^{\eps}\;dx = \int_{\R^2}\partial_tu^\eps\cdot U^\eps\;dx + \int_{\R^2}u^\eps \cdot\partial_tU^\eps\;dx.
    \end{equation}
    Again by Proposition \ref{prop-pressure-time-regularity} and Lemma \ref{prelim-regularity}, we know that for both $u$ and $U$, the momentum equation in \eqref{ans} holds in $L^2(0,T;H^{-1}_x)$ and therefore the convolved counterparts in $L^2(0,T;L^2_x)$. Hence, it is justified to insert them into \eqref{time-product-rule}, which yields together with time integration
    \begin{align*}
        &\int_{\R^2}u^{\eps}(t)\cdot U^{\eps}(t)\;dx - \int_{\R^2}u^{\eps}(s)\cdot U^{\eps}(s)\;dx +2\int_s^t\int_{\R^2}\partial_{x_1}u^\eps\cdot \partial_{x_1} U^\eps\;dx\;d\tau\\&=-\int_s^t\int_{\R^2}\diverg(u\otimes u)^\eps \cdot U^\eps\;dx\;d\tau-\int_s^t\int_{\R^2}\diverg(U\otimes U)^\eps \cdot u^\eps\;dx\;d\tau,
    \end{align*}
     for $0\leq s \leq t \leq T$. The pressure term vanishes, as the convolved velocities remain divergence-free. The pointwise $L^2_x$-convergence of the convolution is sufficient for the product terms to pass to the limit pointwise in $L^1_x$. Concerning the trilinear forms, we need a closer look. As our mollifier is even, it is self-adjoint, which yields together with incompressibility and integration by parts that
    \begin{align*}
        \int_{\R^2}\diverg(u\otimes u)^\eps \cdot U^\eps\;dx&=\int_{\R^2}\diverg(u\otimes u) \cdot (U^\eps)^\eps\;dx=-\int_{\R^2}u\otimes u: \nabla (U^\eps)^\eps\;dx\\&=-\int_{\R^2}(u_1^2-u_2^2)\partial_{x_1}(U_1^\eps)^\eps+u_1u_2(\partial_{x_2}(U_1^\eps)^\eps+\partial_{x_1}(U_2^\eps)^\eps)\;dx\\
        &=\int_{\R^2}2u_1\partial_{x_1}u_1(U_1^\eps)^\eps+u_2^2\partial_{x_1}(U_1^\eps)^\eps-u_1u_2(\partial_{x_2}(U_1^\eps)^\eps+\partial_{x_1}(U_2^\eps)^\eps)\;dx
    \end{align*}
    and 
    \begin{equation*}
        \int_{\R^2}\diverg(U\otimes U)^\eps \cdot u^\eps\;dx=\int_{\R^2}\diverg(U\otimes U) \cdot (u^\eps)^\eps\;dx=\int_{\R^2}(U\cdot\nabla U)\cdot(u^\eps)^\eps\;dx.
    \end{equation*}
    For a clearer notation, we write $v^\eps$ instead of $(v^\eps)^\eps$. With similar computations as in Lemma \ref{prelim-regularity}, we can pass to the limit in
    \begin{align*}
        &\int_s^t\int_{\R^2}\abs{2u_1\partial_{x_1}u_1(U_1^\eps-U_1)+u_2^2\partial_{x_1}(U_1^\eps-U_1)-u_1u_2(\partial_{x_2}(U_1^\eps -U_1)+\partial_{x_1}(U_2^\eps-U_2))}\;dx\;d\tau\\&\leq2\int_s^t\norm{u_1}_{L^\infty_{x_1}L^2_{x_2}}\norm{U_1^\eps-U_1}_{L^2_{x_1}L^\infty_{x_2}}\norm{\partial_{x_1}u_1}_{L^2_{x}}\;d\tau+\left(\norm{u_2^2}_{L^2_tL^2_{x}}+\norm{u_1u_2}_{L^2_tL^2_{x}}\right)\norm{\nabla U^\eps-\nabla U}_{L^2_tL^2_{x}}\\&\leq C\norm{u_1}_{L^\infty_tL^2_{x}}^{1/2}\norm{\partial_{x_1}u_1}_{L^2_tL^2_{x}}^{3/2}\norm{U_1^\eps-U_1}_{L^\infty_tL^2_{x}}^{1/2}\norm{\partial_{x_2}U_1^\eps-\partial_{x_2}U_1}_{L^2_tL^2_{x}}^{1/2}\\&+\left(\norm{u_2^2}_{L^2_tL^2_{x}}+\norm{u_1u_2}_{L^2_tL^2_{x}}\right)\norm{\nabla U^\eps-\nabla U}_{L^2_tL^2_{x}}\to 0,
    \end{align*}
    and
    \begin{align*}
        \int_s^t\int_{\R^2}\abs{(U\cdot\nabla U)\cdot(u^\eps-u)}\;dx\;d\tau&\leq \int_s^t\norm{\nabla U}_{L^2_x}\norm{U}_{L^2_{x_1}L^\infty_{x_2}}\norm{u^\eps-u}_{L^\infty_{x_1}L^2_{x_2}}\;d\tau\\
        &\leq\int_s^t\norm{\nabla U}_{L^2_x}^{3/2}\norm{U}_{L^2_{x}}^{1/2}\norm{u^\eps-u}_{L^2_x}^{1/2}\norm{\partial_{x_1}u^\eps-\partial_{x_1}u}_{L^2_x}^{1/2}\;d\tau\\
        &\leq \norm{\nabla U}_{L^2_tL^2_x}^{3/2}\norm{U}_{L^\infty_tL^2_{x}}^{1/2}\norm{u^\eps-u}_{L^\infty_tL^2_x}^{1/2}\norm{\partial_{x_1}u^\eps-\partial_{x_1}u}_{L^2_tL^2_x}^{1/2}\\
        &\to 0.
    \end{align*}
    Note that we here relied on the additional regularity $\partial_{x_2}U_1\in L^2_tL^2_x$. In total, we arrive at 
    \begin{equation}\label{cross-terms}
        \begin{aligned}
            &2\int_{\R^2}u(t)\cdot U(t)\;dx - 2\int_{\R^2}u(s)\cdot U(s)\;dx +4\int_s^t\int_{\R^2}\partial_{x_1}u\cdot \partial_{x_1} U\;dx\;d\tau\\&=-2\int_s^t\int_{\R^2}(u\cdot\nabla u) \cdot U\;dx\;d\tau-2\int_s^t\int_{\R^2}(U\cdot\nabla U) \cdot u\;dx\;d\tau.
        \end{aligned}
    \end{equation}
    Adding both energy equalities for $u$ and $U$, and subtracting \eqref{cross-terms} yields
    \begin{equation}\label{w-energy-balance}
        \norm{w(t)}_{L^2_x}^2+2\int_s^t\norm{\partial_{x_1}w}_{L^2_x}^2\;d\tau=\norm{w(s)}_{L^2_x}^2-2\int_s^t\int_{\R^2}(w\cdot\nabla U)\cdot w\;dx\;d\tau,
    \end{equation}
    where we used the identities
    \begin{equation*}
        \int_{\R^2}(u\cdot\nabla u)\cdot U\;dx=-\int_{\R^2}(u\cdot\nabla U)\cdot u\;dx,\quad \int_{\R^2}(w\cdot\nabla U)\cdot U\;dx=\frac{1}{2}\int_{\R^2}w\cdot\nabla\abs{U}^2\;dx=0.
    \end{equation*}
    We want to conclude with Gr\"onwall's lemma, therefore we need to bound the second term on the right hand-side of \eqref{w-energy-balance}. We use the same estimate as above, the argument of Lemma \ref{prelim-regularity}, and Young's inequality to get
    \begin{align*}
        &\abs{\int_s^t\int_{\R^2}(w\cdot\nabla U)\cdot w\;dx\;d\tau}\\&\leq \int_s^tC\norm{w_1}_{L^2_{x}}^{1/2}\norm{\partial_{x_1}w_1}_{L^2_{x}}^{3/2}\norm{U_1}_{L^2_{x}}^{1/2}\norm{\partial_{x_2}U_1}_{L^2_{x}}^{1/2}+\left(\norm{w_2^2}_{L^2_{x}}+\norm{w_1w_2}_{L^2_{x}}\right)\norm{\nabla U}_{L^2_{x}}\;d\tau\\
        &\leq \int_s^tC\norm{w_1}_{L^2_{x}}^{1/2}\norm{\partial_{x_1}w_1}_{L^2_{x}}^{3/2}\norm{U_1}_{L^2_{x}}^{1/2}\norm{\partial_{x_2}U_1}_{L^2_{x}}^{1/2}\\&+2\left(\norm{w_2}_{L^2_{x}}\norm{\partial_{x_1}w}_{L^2_{x}}+\norm{w_1}_{L^2_{x}}^{1/2}\norm{w_2}_{L^2_{x}}^{1/2}\norm{\partial_{x_1}w}_{L^2_{x}}\right)\norm{\nabla U}_{L^2_{x}}\;d\tau\\
        &\leq \int_s^t \frac{1}{6}\norm{\partial_{x_1}w_1}_{L^2_{x}}^{2} + C\norm{w_1}_{L^2_{x}}^{2}\norm{U_1}_{L^2_{x}}^{2}\norm{\partial_{x_2}U_1}_{L^2_{x}}^{2} + \frac{1}{6}\norm{\partial_{x_1}w}_{L^2_{x}}^{2} + C\norm{w_2}_{L^2_{x}}^2\norm{\nabla U}_{L^2_{x}}^2\\
        &+ \frac{1}{6}\norm{\partial_{x_1}w}_{L^2_{x}}^{2}+C\norm{w_1}_{L^2_{x}}\norm{w_2}_{L^2_{x}}\norm{\nabla U}_{L^2_{x}}^2\;d\tau\\
        &\leq \frac{1}{2}\int_s^t\norm{\partial_{x_1}w}_{L^2_{x}}^{2}\;d\tau+C\int_s^t\norm{w}_{L^2_{x}}^{2}\left(1+\norm{U}_{L^2_{x}}^{2}\right)\norm{\nabla U}_{L^2_{x}}^{2}\;d\tau.
    \end{align*}
    Going back to \eqref{w-energy-balance}, we obtain
    \begin{equation*}
        \norm{w(t)}_{L^2_x}^2+\int_s^t\norm{\partial_{x_1}w(\tau)}_{L^2_x}^2\;d\tau\leq\norm{w(s)}_{L^2_x}^2+\int_s^tA(\tau)\norm{w(\tau)}_{L^2_x}^2\;d\tau,
    \end{equation*}
    with $A(\tau):=C\left(1+\norm{U(\tau)}_{L^2_{x}}^{2}\right)\norm{\nabla U(\tau)}_{L^2_{x}}^{2}\in L^1(0,T)$. Hence, we can apply Gr\"onwall's lemma for $s=0$ with $w(s)=0$ to get
    \begin{equation*}
        \norm{w(t)}_{L^2_x}^2=0 \text{ for all }t\in[0,T].
    \end{equation*}
\end{proof}
\appendix
\section{Auxiliary analytic results}
We use the unitary Fourier transform
\begin{equation}\label{fourier}
    \widehat f(\xi):=\frac{1}{2\pi}
\int_{\mathbb{R}^2}e^{-ix\cdot\xi}f(x)\;dx.
\end{equation}
\begin{lem}[Riesz transform]\label{riesz}
    Let $f\in\mathcal{S}(\R^2)$. The $i$-th Riesz transform, defined by
    \begin{equation*}
        \widehat{R_if}(\xi):=-i\frac{\xi_i}{\abs{\xi}}\widehat{f(\xi)}, \text{ for }\xi\neq0,
    \end{equation*}
    extends to a bounded operator on $L^p_x$ for every $p\in(1,\infty)$,
    \begin{equation*}
        \norm{R_if}_{L^p_x}\leq C_p\norm{f}_{L^p_x},.
    \end{equation*}
    with $C_2=1$.
\end{lem}
\begin{proof}
    See \cite[Section 5.1]{grafakos}.
\end{proof}
\begin{lem}[Gagliardo--Nirenberg inequality]\label{gagliardo}
    Let $f\in H^1(\R)$, and $q\in[2,\infty]$. Then there exists a constant $C_q>0$ such that
    \begin{equation*}
        \norm{f}_{L^q(\R)}\leq C_q\norm{f}_{L^2(\R)}^{1/2+1/q}\norm{f'}_{L^2(\R)}^{1/2-1/q}, \text{ where }1/\infty=0.
    \end{equation*}
\end{lem}
\begin{proof}
    See \cite[Lecture II, Theorem, p. 125]{nirenberg}.
\end{proof}

\end{document}